\RequirePackage{fix-cm}
\documentclass[smallextended]{svjour3}       
\smartqed  
\usepackage{graphicx}
\usepackage{amsfonts}
\usepackage{mathrsfs}
\usepackage{amsmath}
\usepackage [latin1]{inputenc}
\begin{document}

\title{Smooth distributions on subcartesian
spaces are globally finitely generated
}


\author{Qianqian Xia          
}


\institute{Qianqian Xia \at
           Nanjing University of Information Science \& Technology, Nanjing, 210044, China\\
              \email{inertialtec@sina.com}
}

\date{Received: date / Accepted: date}

\maketitle

\begin{abstract}
We prove that a connected subcartesian space admits embedding in a Euclidean space.  The Whitney Embedding Theorem is then stated as a corollary of our result. Based on the above result together with the theory of distribution on smooth manifolds,  we show that smooth generalized distributions on connected subcartesian spaces are globally finitely generated. We also show that smooth generalized subbundles of vector bundles on connected subcartesian spaces are globally finitely generated.
\keywords{ Subcartesian space\and Embedding and Generalized distribution}
\end{abstract}

\section{Introduction}
\label{intro}
Distribution theory has its application in many fields such as geometric mechanics and control \cite{11,1}. Given a smooth manifold $M$, by  assigning to each point $x\in M$ a subspace of $T_xM$, we get a generalized distribution $D$. Sussmann and Stefan considered a generalized distribution $D$ to be smooth if for each $x\in M$, there are locally defined smooth vector fields that are sections of $D$ whose values at $x$ span $D_x$. If $E$ is a vector bundle over $M$, a generalized subbunlde $F$ of $E$ is an assignment $x\rightarrow F_x$ of a subspace of $F_x$ of the fiber $E_x$ of $E$ over $x$, for each $x \in M$. $F$ is said to be smooth if for each $x \in M$, there are locally defined smooth sections of $F$ whose value at $x$ span $F_x$.

In \cite{1},  Lewis and Bullo give an interesting discussion of generalized distributions and generalized subbundles. A question whether a generalized subbundle has a finite set of global generators is proposed there.

In \cite{3}, Drager et.al. show that any smooth generalized distribution on a smooth manifold is globally finitely generated, that is, there
are finitely many globally defined sections whose values span the generalized distribution at each point. The result in \cite{3} provides a complete answer to the conjecture which was first proposed in \cite{1}. A more recently detailed discussion on this topic can be found in \cite{2}.

The existing results on distribution theory are restricted to smooth manifolds, since differential geometry was traditionally regarded as the study of smooth manifolds. However,  there has long been perceived the need for an extension of the framework of smooth manifolds in differential geometry,
which is too restrictive and does not admit certain basic geometric intuitions. Sikorski¡¯s theory of differential spaces \cite{6,7} studies the differential geometry of a large class of singular spaces which both contains the theory of manifolds and allows the investigation of singularities. It is the investigation of geometry in terms of differentiable functions. A similar expansion occurred in the realm of algebra geometry, which is the investigation of geometry in terms of polynomials. The  difference between the two theories is in the choice of the space of functions.  A differential space $S$ is said to be subcartesian \cite{4, 5} if every point of $S$ has a neighbourhood diffeomorphic to a subset of some Cartesian space $\mathbb{R}^n$. It  has been shown in \cite{5} that the orbit space of a proper action of a connected Lie group on a smooth manifold is a subcartesian space. The theory of subcartesian spaces has been developed by \'Sniatycki et.al. in recent few years \cite{12, 13, 14, 15, 16}. See \cite{5} for a systematic treatment on this topic.

 Given a differential space $S$, the tangent space $T_xS$ of $S$ is defined as the set of all derivations of $C^{\infty}(S)$ at $x\in S$ and the tangent bundle $TS$ is the union  of all tangent spaces to $S$ at all points in $S$. $TS$ is a differential space with the differential structure $C^{\infty}(TS)$ inherited from $C^{\infty}(S)$.

Analogous to the terminology in \cite{1,3,2}, a generalized distribution on a subcartesian space $S$ is a subset $D \subseteq TS$ such that for each $x \in S$, the intersection $D_x=D \cap T_xS$ is a vector subspace of $T_xS$. A generalized distribution on  $S$ is smooth, if for any $v_x \in T_xS$, there exist a  neighborhood $U$ of $x$ and a local smooth section $X \in \Gamma(TS;U)$ such that $X(x)=v_x$ and $X(y)\in D(y)$ for each $y \in U$. We say that a generalized distribution $D$ on $S$ is globally finitely generated, if there exist derivations $X_1, \cdots, X_l$ of $C^{\infty}(S)$ such that $X_1(x), \cdots, X_l(x)$ span $D_x$ for each $x \in S$.

Analogous to the definition of vector bundle on smooth manifolds, a vector bundle $E$ on a subcartesian space $S$ is a smooth map between subcartesian spaces which has the local product property with respect to $\mathbb{R}^k$. Besides, for each $x\in S$, fiber $E_x$ is a $k$-dimensional real vector space. A generalized subbundle $F$ of a vector bundle  $E$ on a subcartesian space $S$ is a subset of $E$, such that $F_x=E_x\cap F$ is a subspace of $E$. A generalized subbundle $F$  of $E$ is smooth, if for any $e_x \in F_x$, there exist a neighborhood $U$ of $x$ and a local smooth section $\sigma: U\rightarrow E$ such that $\sigma(x)=e_x$ and $\sigma(y)\in F_y$ for each $y \in U$. We say that a smooth generalized subbundle $F$ on $S$ is globally finitely generated, if there exist global smooth sections $\sigma_1, \cdots, \sigma_l$ of $F$ such that $\sigma_1(x), \cdots, \sigma_l(x)$ span $F_x$ for each $x \in S$.

In this paper, we study distribution theory on subcartesian space. We first present embedding theorem for subcartesian space.  We show that every connected subcartesian space admits a smooth embedding into a Euclidean space, where the smoothness is within the category of differential spaces.
We provide a complete proof  which has not been found in the existing literature. We then present the Whitney Embedding Theorem as a corollary of our embedding theorem here. Based on this, together with Theorem 3.1.6 in \cite{5} and the main theorem in \cite{3}, we show that every smooth generalized distribution on a connected subcartesian space is globally finitely generated. We also show that  every smooth generalized subbundle on a connected subcartesian space is globally finitely generated.
\section{Preliminaries}
\label{sec:1}
\begin{definition}\label{def3}
A differential structure on a topological space $S$ is a family $C^{\infty}(S)$ of real-valued functions on $S$ satisfying the following conditions:

1. The family
\[\{f^{-1}(I)|f\in C^{\infty}(S) \text{and} \, I \, \text{is an open interval in} \, \mathbb{R}\}\]
is a subbasis for the topology of $S$.

2. If $f_1, \cdots, f_n \in C^{\infty}(S)$ and $F \in C^{\infty}(\mathbb{R}^n)$, then $F(f_1, \cdots, f_n) \in C^{\infty}(S)$.

3. If $f: S\rightarrow \mathbb{R}$ is a function such that, for every $x \in S$, there exist an open neighborhood $U$ of $x$,
and a function $f_x \in C^{\infty}(S)$ satisfying
\[f_x|U=f|U,\]
then $f \in C^{\infty}(S)$. Here, the subscript vertical bar $|$ denotes a restriction.
\end{definition}
\begin{proposition}\label{prop1}
Let $S$ be a differential space. Then for any open set $U \subseteq S$, for each $x \in U$, there exists an open subset $V \subseteq U$ that contains $x$ such that $cl(V) \subseteq U$. Here $cl(V)$ denotes the closure of the subset $V$ in $S$.
\end{proposition}
\begin{proof}
According to condition 1 in Definition \ref{def3}, for any open set $U \subseteq S$ and any $x \in U$, there exist $f_1, \cdots, f_k \in C^{\infty}(S)$ and intervals
$(a_1, b_1), \cdots, (a_k,b_k)$ in $\mathbb{R}$, such that $f_1^{-1}(a_1,b_1)\cap\cdots\cap f_k^{-1}(a_k,b_k) \subseteq U$ contains $x$. Let $[c_1, d_1] \subseteq (a_1, b_1), \cdots, [c_k,d_k] \subseteq (a_k, b_k)$ be such that $f_1(x)\in (c_1,d_1), \cdots, f_k(x) \in (c_k, d_k)$.
Consider the subset $V=f_1^{-1}(c_1,d_1)\cap\cdots\cap f_k^{-1}(c_k,d_k)$. Then $x \in V$. It follows from condition 1 in Definition \ref{def3} that $V$ is open. Beside, $cl(V)=f_1^{-1}[c_1,d_1]\cap\cdots\cap f_k^{-1}[c_k,d_k] \subseteq f_1^{-1}(a_1,b_1)\cap\cdots\cap f_k^{-1}(a_k,b_k) \subseteq U$. Hence the result follows immediately.
\end{proof}
\begin{definition}
A map $\phi: R \rightarrow S$ is $C^{\infty}$ if $\phi^*f=f\circ \phi \in C^{\infty}(R)$ for
every $f \in C^{\infty}(S)$. A $C^{\infty}$ map $\phi$ between differential spaces is a $C^{\infty}$ diffeomorphism
if it is invertible and its inverse is $C^{\infty}$.
\end{definition}
\begin{definition}\label{def2}
A differential space $S$ is subcartesian if it is Hausdorff and every point $x \in S$ has a  neighbourhood $U$ diffeomorphic to a subset $V$ of $\mathbb{R}^n$ by the map $\phi: U\rightarrow \mathbb{R}^n$. $(U, \mathbb{R}^n, \phi)$ is said to be a local chart of $S$.  The subcartesian space
is said to be locally connected if for every point $x \in S$, there exists a connected neighborhood $U$ such that $(U, \mathbb{R}^n, \phi)$ is a local chart of $S$.
\end{definition}
\begin{definition}
Let $S$ be a subcartesian space. The structural dimension of $S$ at a point $x\in S$ is the smallest integer $n_x$ such that for some
open neighborhood $U \subseteq S$ of $x$, there is a diffeomorphism of $U$ onto a subset $V \subseteq \mathbb{R}^{n_x}$. The structural dimension of
$S$ is the smallest integer $n$ such that for every point $x \in S$, the structural dimension $n_x$ of $S$ at $x$ satisfies that $n_x\leq n$.
\end{definition}
\begin{definition}
Let $f\in C^{\infty}(S)$. For each $x \in S$,  the smooth function $g$  on $\mathbb{R}^n$ is said to be a smooth extension of $f$ at $x \in S$, if
there exists a local chart $(U, \mathbb{R}^n, \phi)$  of $S$ with $x \in U$ such that $g\circ \phi|U=f|U$.
\end{definition}
Given any smooth function on $S$ and any $x \in S$, the existence of smooth extension is ensured by the above definition of subcartesian space.
\begin{lemma}\cite{5}
For every open subset $U$ of a differential space $S$ and every $x \in U$, there exists $f \in C^{\infty}(S)$
satisfying $f|V=1$ for some neighbourhood
$V$ of $x$ contained in $U$, and $f|W=0$ for some open subset $W$ of $S$ such that
$U\cup W=S$.
\end{lemma}
\begin{lemma}\cite{5}\label{lem13}
Let $S$ be locally compact, Hausdorff and second countable. Then every open cover $\{U_\alpha\}$ of $S$ has a countable, locally finite refinement
consisting of open sets with compact closures.
\end{lemma}
\begin{definition}
A  countable partition of unity on a differential space $S$ is a countable family of functions $\{f_i\} \in C^{\infty}$
such that:\\
(a) The collection of their supports is locally finite.\\
(b) $f_i(x)\geq 0$ for each $i$ and each $x \in S$.\\
(c) $\sum_{i=1}^{\infty}f_i(x)=1$ for each $x \in S$.
\end{definition}
\begin{theorem}\cite{5}\label{thm2}
Let $S$ be a differential space with differential structure $C^{\infty}(S)$,
and let $\{U_\alpha\}$ be an open cover of $S$. If $S$ is Hausdorff, locally compact and second
countable, then there exists a countable partition of unity $\{f_i\} \in C^{\infty}(S)$,
subordinate to $\{U_\alpha\}$  and such that the support of each $f_i$ is compact.
\end{theorem}
\begin{definition}
A derivation of $C^{\infty}(S)$ is a linear map
\[X: C^{\infty}(S)\rightarrow C^{\infty}(S) : f \rightarrow X(f)\]
satisfying Leibniz¡®s rule
\[X(f_1f_2)=X(f_1)f_2+f_1X(f_2)\]
for every $f_1, f_2\in C^{\infty}(S)$.
\end{definition}
We denote by $\text{Der}C^{\infty}(S)$ the space of derivations of $C^{\infty}(S)$.
\begin{definition}
A derivation of $C^{\infty}(S)$ at $x \in S$ is a linear map $v:C^{\infty}(S)\rightarrow \mathbb{R}$ such that
\[v(f_1f_2)=v(f_1)f_2(x)+f_1(x)v(f_2)\]
for every $f_1, f_2\in C^{\infty}(S)$.
\end{definition}
We interpret derivations of $C^{\infty}(S)$ at $x \in S$ as tangent vectors to $S$ at $x$. The
set of all derivations of $C^{\infty}(S)$ at $x$ is denoted by $T_xS$ and is called the tangent
space to $S$ at $x$.
\begin{lemma}\cite{5}\label{lem15}
If $f\in C^{\infty}(S)$ vanishes identically in an open set $U \subseteq S$, then
$X(f)|U=0$ for all $X\in \text{Der}C^{\infty}(S)$.
\end{lemma}
\begin{definition}
The tangent bundle of a differential space $S$ is the union $TS$ of all tangent spaces to $S$ at all points in $S$. In other words,
$TS=\cup_{x\in S}T_xS$. The tangent bundle projection is the map $\tau:TS\rightarrow S$ that assigns to each $v \in TS$ the point $x \in S$
such that $v \in T_xS$.
\end{definition}
\begin{definition}
The differential structure $C^{\infty}(TS)$ of $TS$ is generated by the family of functions $\{f\circ \tau,df|f\in C^{\infty}(S)\}$.
\end{definition}
\begin{definition}
 A smooth section of $\tau:TS\rightarrow S$ is a smooth map $\sigma: S\rightarrow TS$ such that $\tau\circ \sigma=\text{identity}$.
 A local smooth section of $\tau:TS\rightarrow S$ is a smooth map $\gamma: U \rightarrow \tau^{-1}(U)$ such that $\tau\circ \gamma=\text{identity}$.
\end{definition}
\begin{definition}\label{def1}
 A generalized distribution on a subcartesian space $S$ is a subset $D \subseteq TS$ such that for each $x \in S$, the intersection
$D_x=D \cap T_xS$ is a vector subspace of $T_xS$. $D$ is smooth, if for any $v_x \in D_x$, there exists a  neighborhood $U$ of $x$ and a local smooth section $X: U\rightarrow \tau^{-1}(U)$ such that $X(x)=v_x, X(y)\in D_y$, for each $y \in U$.
\end{definition}

Let $\pi:E \rightarrow S$ be a smooth map between subcartesian spaces. The map $\pi$ will be said to have the local product property with respect to $\mathbb{R}^k$, if there is an open covering $\{U_{\alpha}\}$ of $S$ and a family $\{\psi_{\alpha}\}$ of diffeomorphisms
\[\psi_{\alpha}: U_{\alpha}\times \mathbb{R}^k \rightarrow \pi^{-1}(U_{\alpha})\]
such that $\pi\circ \psi_{\alpha}(x,y)=x$, for each $x\in U_{\alpha}, y \in \mathbb{R}^k$. The system $\{(U_{\alpha}, \psi_{\alpha})\}$ will be called a local decomposition of $\pi$.
\begin{definition}
A smooth vector bundle is a four-tuple $(E, \pi, S, \mathbb{R}^k)$ where $\pi: E \rightarrow S$ is a smooth map between subcartesian spaces which has the local  product property with respect to $\mathbb{R}^k$. Besides, for each $x \in S$, $\pi^{-1}(x)$ is a $k$-dimensional real vector space, and $\psi_{\alpha}^{-1}|\pi^{-1}(x): \pi^{-1}(x)\rightarrow \{x\}\times \mathbb{R}^k$ is an linear isomorphism, for each  $x \in U_\alpha$.
A local decomposition for $\pi$ is called a coordinate representation for the vector bundle.
\end{definition}
\begin{definition}
A smooth section of $(E, \pi, S, \mathbb{R}^k)$ is a smooth map $\sigma: S\rightarrow E$ such that $\pi\circ \sigma=\text{identity}$.
 A local smooth section of $(E, \pi, S, \mathbb{R}^k)$ is a smooth map $\gamma: U \rightarrow \pi^{-1}(U)$ such that $\pi\circ \gamma=\text{identity}|U$, where $U$ is an open subset on $S$.
\end{definition}
\begin{definition}
A generalized subbundle of a smooth vector bundle  $(E, \pi, S, \mathbb{R}^k)$ is a subset $F \subseteq E$ such that for each $x \in S$,
$F_x=\pi^{-1}(x)\cap F$ is a subspace of $\pi^{-1}(x)$. $F$ is smooth, if for any $e_x \in F_x$, there exist a  neighborhood $U$ of $x$ and a local smooth section $\sigma: U\rightarrow \pi^{-1}(U)$ such that $\sigma(x)=e_x$, and $\sigma(y) \in F_y$, for each $y \in U$.
\end{definition}
\begin{definition}
A smooth Riemannian metric in $(E, \pi, S, \mathbb{R}^k)$ is a symmetric positive definite bilinear form $g(x)$ in $\pi^{-1}(x)$, for each
$x \in S$, such that for each smooth section $\sigma$ of $(E, \pi, S, \mathbb{R}^k)$, the function $g(x)(\sigma(x),\sigma(x))\in C^{\infty}(S)$.
\end{definition}
\begin{proposition}\label{prop4}
Each smooth vector bundle admits a smooth Riemannian metric.
\end{proposition}
\begin{proof}
Consider the trivial bundle $S \times \mathbb{R}^k$. Let $<,>$ be a Euclidean metric in $\mathbb{R}^k$, then
\[g(x)(y_1, y_2)=<y_1, y_2>, x\in S, y_1,y_2 \in \mathbb{R}^k\]
defines a smooth Riemannian metric in $S \times \mathbb{R}^k$.

Now consider any smooth vector bundle $(E, \pi, S, \mathbb{R}^k)$. Let $\{(U_{\alpha}, \psi_{\alpha})\}$ be a local decomposition of $\pi$. It follows from Theorem \ref{thm2} that there exists a subordinate partition of unity $\{p_{\alpha}\}$.

Since the restriction of $(E, \pi, S, \mathbb{R}^k)$ to $U_{\alpha}$ is trivial, there is a Riemannian metric $g_{\alpha}$. Define $g$ by
$\Sigma_{\alpha}p_{\alpha}g_{\alpha}$. Then $g(x)$ is a Euclidean metric in $\pi^{-1}(x)$; hence $g$ is a smooth Riemannian metric in $(E, \pi, S, \mathbb{R}^k)$.
\end{proof}
In the following we restrict our attention to Hausdorff, locally compact, second countable subcartesian spaces. Hence, we shall be able to rely on
the existence of partitions of unity. Note that a subcartesian space is Hausdorff;
see Definition \ref{def2}. Moreover, Definition \ref{def2} and condition 1 of Definition
\ref{def3} imply that a subcartesian space is locally compact. Thus, we make
the following assumption.

{\bfseries{Assumption}}. Throughout the paper, all subcartesian spaces are second countable and connected.

\section{Dimension Theory}
\begin{definition}\cite{9}
A space $A$ has covering dimension at most $m$
if every open covering of $A$ has a refinement of order at most $m$ (which
means that no point of $A$ lies in more than $m+1$ elements of the refinement).
If $A$ is a closed subset of the topological space $M$, this is equivalent to the statement
that every open covering of $M$ has a refinement whose restriction to
$A$ has order at most $m$. If $A'$ is a closed subset of $A$ and $A$ has
covering dimension at most $m$, so does $A'$.
\end{definition}
\begin{lemma}\label{lem1}\cite{8}
Let $M$ be a second countable locally compact Hausdorff topological space. Then there exist countable many open sets $G_1, G_2, \cdots, G_k, \cdots$
satisfying

(1) $cl(G_j)$ is compact, $j=1,2, \cdots$;

(2) $cl(G_j)\subseteq G_{j+1}, j=1,2, \cdots$;

(3) $\cup G_j=\cup cl(G_j)=M,$\\
where $cl(G_j)$ denotes the closure of $G_j,j=1,2,\cdots$.
\end{lemma}
\begin{lemma}\label{lem2}
Let $S$ be a subcartesian space and $\Phi \in C^{\infty}(S)$.
Then there exist locally finite open covers  $(U_j)_{j\in \mathbb{Z}_{>0}}, (V_j)_{j \in \mathbb{Z}_{>0}}, (W_j)_{j \in \mathbb{Z}_{>0}}$ such that $cl(U_j)\subseteq V_j, cl(V_j)\subseteq W_j$, and $cl(W_j)$ is compact, for
each $j>0$, where $(W_j,\mathbb{R}^{n_j}, \phi_j)$ is a local chart of $S$, such that there exists a smooth extension $f_j$ of $\Phi$ on $W_j$, that is, $f_j\circ \phi_j|W_j=\Phi|W_j$.
\end{lemma}
\begin{proof}
From Lemma \ref{lem1} we know that there exist countable open sets $G_1, \cdots, G_k,\cdots$ satisfying conditions (1), (2) and (3) in Lemma \ref{lem1}.
It follows that $cl(G_h)/{G_{h-1}}$ is compact,  $G_{h+1}/{cl (G_{h-2})}$ is open and $cl(G_h)/{G_{h-1}}\subseteq G_{h+1}/{cl (G_{h-2})}$. Since $S$ is locally compact, for any $p\in cl({G}_h)/{G_{h-1}}$, there exists a compact set $C$ in $S$ containing an open neighborhood $U$ of $p$. On the other hand we know the  local charts of $S$ form an open cover of $S$. Then for $p\in cl({G}_h)/{G_{h-1}}$, there exist a local chart $(\mathscr{V}, \phi)$ and an open set $W\subset \mathscr{V}$ such that

(1) There exists a smooth function $f$ on $\mathbb{R}^n$ satisfying $f\circ \phi|W=\Phi|W$;

(2) $p\in W\subseteq (G_{h+1}/{cl({G}_{h-2})})\cap \mathscr{V}$;

(3) $\phi(p)=0$ and $\phi(W)=\phi(\mathscr{W})\subseteq \phi(\mathscr{V})$; \\
Here $\mathscr{W}$ is an open subset containing $p$ such that $cl(\mathscr{W}) \subseteq (G_{h+1}/{cl({G}_{h-2})})\cap \mathscr{V}\cap U$. The existence of $\mathscr{W}$ is ensured by Proposition \ref{prop1}. Since $cl(W) \subseteq U\subseteq C$ and $C$ is compact, it follows that
$cl(W)$ is compact.

Let $\mathscr{W}_1$ be an open set containing $p$ such that $cl(\mathscr{W}_1)\subseteq W$.
Denote by $V=\mathscr{W}_1$.   And  let $\mathscr{W}_2$ be an open set containing $p$ such that $cl(\mathscr{W}_2)\subseteq V$. Denote by $U=\mathscr{W}_2$.  Then we have $cl(U) \subseteq V$ and $cl(V)\subseteq W$.

Since  ${cl({G}_h})/{G_{h-1}}$ is compact, there exist finitely many points $p_{h, 1}, p_{h, 2}, \cdots$, $p_{h, k_h} \in cl({G}_h)/{G_{h-1}}$, such that the corresponding open sets $U_{h,1}, U_{h,2}, \cdots, U_{h,k_h}$  form an open cover of $cl(G_h)/G_{h-1}$.
We claim that the corresponding open sets
\[(U_{1,1}, U_{1,2}, \cdots, U_{1,k_1}; U_{2,1}, U_{2,2}, \cdots, U_{2, k_2}; \cdots),\]
\[(V_{1,1}, V_{1,2}, \cdots, V_{1,k_1}; V_{2,1}, V_{2,2}, \cdots, V_{2, k_2}; \cdots),\]
\[(W_{1,1}, W_{1,2}, \cdots, W_{1,k_1}; W_{2,1}, W_{2,2}, \cdots, W_{2, k_2}; \cdots),\]
satisfy the conditions in the lemma.   We only need to prove the local finiteness of $(W_{i,j})$. Given $p\in S$, assume that
$p \in G_r$, then it follows from the above construction that there exist finite many $W_{i,j}$ that intersect $G_r$. In fact,
\[W_{i,j}\cap G_r=\emptyset, i\geq r+2, 1\leq j \leq k_i.\]
This completes the proof of the lemma.
\end{proof}

\begin{lemma}\cite{9}\label{lem9}
Let $A_n$ be a closed subset of $M$, such that $A_n\subseteq
\text{Int}A_{n+1}$ and $\cup A_n=M$. If $A_l$ and each set $cl(A_{n+1} - A_n)$ have covering
dimension at most $m$, so does $M$.
\end{lemma}
\begin{theorem}
Every subcartesian space $S$ has covering dimension at most $n$, where $n$ is the structural dimension of $S$.
\end{theorem}
\begin{proof}
Cover $S$ with the sets $\{cl(V_j)\}_{j \in \mathbb{Z}_{>0}}$ as given in Lemma \ref{lem2}. Let $A_1=cl(V_1)$.  In general, given
$A_l=cl(V_1)\cup\cdots\cup cl(V_q)$, since $cl(V_1)\cup\cdots\cup cl(V_q)$ is compact, we can choose an integer $p>q$ such that $A_l\subseteq \text{Int}(cl(V_1)\cup\cdots\cup cl(V_p))$, and let $A_{l+1}=
cl(V_1)\cup\cdots\cup cl(V_p)$.  Since $\phi_j(cl(V_j))$ is compact in $\mathbb{R}^{n_j}$, it follows from Theorem 2.15 in \cite{9} that $\phi_j(cl(V_j))$ has covering dimension at most $n_j$, which yields that $cl(V_j)$ has covering dimension at most $n_j$ since $\phi_j$ is a diffeomorphism. Hence according to Corollary 2.14 in \cite{9}, each $A_l$ has covering dimension at most $n$. Then the result follows immediately from Lemma \ref{lem9}.
\end{proof}

\begin{lemma}\label{lem10}
Let $S$ be a subcartesian space with structural dimension $n$. Let $\mathcal{O}$ be an open covering of $S$. There is a locally finite open refinement of $\mathcal{O}$  which is the union
of $n+1$ collections $\mathcal{O}_0, \cdots, \mathcal{O}_m$ such that the sets belonging to any one $\mathcal{O}_i$
are pairwise disjoint.
\end{lemma}
\begin{proof}
According to Lemma \ref{lem13}, let $(U_j)_{j \in \mathbb{Z}>0}$ be a countable locally finite open refinement of $\mathcal{O}$. It follows from Theorem \ref{thm2} that there exists a partition of unity $(\phi_j){_{j \in \mathbb{Z}>0}}$ such that
the support of $\phi_j$ is contained in $U_j$, for each $j$. Then by using the same arguments as given in the proof of Lemma 2.7 in \cite{9}, the result follows.
\end{proof}

\begin{proposition}
A subcartesian space $S$ admits a finite atlas.
\end{proposition}
\begin{proof}
Cover $S$ with the atlas $\{(U_j, \phi_j)_{j\in \mathbb{Z}_{>0}}\}$  as given in Lemma \ref{lem2}. Then we can choose the mappings $\psi_j:\mathbb{R}^{n_j}\rightarrow \mathbb{R}^n, j\in \mathbb{Z}_{>0}$ such that the sets $\psi_j\circ\phi_j(U_j)$ are disjoint.
Let $\{V_{ij}|i \leq n+1, j \in \mathbb{N}\}$ be the refinement of Lemma \ref{lem10}. Choose $\alpha(i,j)$ such that $V_{ij} \subseteq U_{\alpha(i,j)}$. Now let $V_i=\cup_{j \in \mathbb{N}}V_{ij}$ and define $\upsilon_i: V_i \rightarrow \mathbb{R}^n$ by
$\upsilon_i(x)=\phi_{\alpha(i,j)}(x)$, for $x \in V_{ij}$. Since $V_{ij}$ are disjoint for fixed $i$, and the sets $\psi_j\circ\phi_j(U_j)$ are disjoint, it follows immediately that $\upsilon_i$ is a diffeomorphism between $V_i$ and  a subset of $\mathbb{R}^n$, for $i=1, \cdots, n+1$. Hence $\{(V_i, \upsilon_i),i=1, \cdots, n+1\}$ is a finite atlas for $S$.
\end{proof}

\begin{proposition}\label{prop2}
Every smooth vector bundle $(E,\pi, S, \mathbb{R}^k)$ has a finite coordinate representation.
\end{proposition}
\begin{proof}
Let $\{(U_\alpha, \psi_\alpha)\}$ be any coordinate representation for $(E, \pi, S, \mathbb{R}^k)$. It follows from Lemma \ref{lem10} that there exists a refinement $\{V_{ij}|i=1, \cdots, m, j \in \mathbb{N}\}$ of $\{U_\alpha\}$ such that $V_{ij}\cap V_{ik}=\emptyset$ for $j \neq k$. Let
$V_i=\cup_jV_{ij}$ and define $\psi_i: V_i\times \mathbb{R}^k\rightarrow \pi^{-1}(V_i)$ by $\psi_i(x,y)=\psi_{ij}(x,y)$ if $x \in V_{ij}, y \in \mathbb{R}^k$, where $\psi_{ij}$ is the restriction of some $\psi_\alpha$. It's easy to see that $\psi_i$ is a diffeomorphism. Hence $\{(V_i, \psi_i)|i=1, \cdots, m\}$ is a coordinate representation for the vector bundle.
\end{proof}
\section{Embedding theorem}

\begin{lemma}\label{lem4}\cite{8}
If $\Phi:U\rightarrow \mathbb{R}^m$ is a smooth mapping from an open subset $U \subseteq \mathbb{R}^n$ with $m\geq 2n$,
then for $\epsilon \in \mathbb{R}_{>0}$, there exists $A\in \mathbb{R}^{m\times n}$ such that
(i) $|A_j^i|<\epsilon$ for $i \in \{1, \cdots, m\}, j\in \{1, \cdots, n\}$. (ii) the mapping $x \rightarrow \Phi(x)+Ax$ is an immersion on $U$.
\end{lemma}
Let $U$ be an open set in $\mathbb{R}^n$ and let $K \subseteq U$ be compact in $\mathbb{R}^n$. Consider the smooth mapping
$\Phi: U \rightarrow \mathbb{R}^m$. For $x=(x_1, \cdots, x_n) \in U$,  $\Phi(x)=(\Phi_1(x), \cdots, \Phi_m(x))=(\Phi_1(x_1, \cdots, x_n), \cdots, \Phi_m(x_1, \cdots, x_n))$. Denote by
\[|\Phi|_K^{(0)}=\max_{1\leq i\leq m}\sup_{x\in K}\{|\Phi_i(x)|\},\]
\[|\Phi|_K^{(1)}=\max\{|\Phi|_K^{(0)}, |\frac{\partial \Phi}{\partial x_1}|_K^{(0)}, \cdots, |\frac{\partial \Phi}{\partial x_n}|_K^{(0)}\}.\]
\begin{lemma}\label{lem3}\cite{8}
If $\Phi:U\rightarrow \mathbb{R}^m$ is an immersion on the compact set $K \subseteq U$, where
$U$ is an open subset of $\mathbb{R}^n$,
then there exists $\delta \in \mathbb{R}_{>0}$, such that for any smooth mapping $\Psi: U\rightarrow \mathbb{R}^m$ satisfying that
$|\Psi-\Phi|_K^{(1)}<\delta$, $\Psi$ is an immersion on $K$.
\end{lemma}
\begin{lemma}\label{lem11}
Let $(U, \phi, \mathbb{R}^k)$ and $(V, \psi, \mathbb{R}^m)$ be two local charts of the subcartesian space $S$ such that $U \cap V\neq \emptyset$.
Then for each $x \in U \cap V$, there exists a smooth mapping $F: \mathbb{R}^k \rightarrow \mathbb{R}^m$ such that $F \circ \phi|W=\psi|W$, where $W\subseteq U \cap V$ is a neighborhood of $x$.
\end{lemma}
\begin{proof}
Consider the mapping $\psi \circ \phi^{-1}: \phi(U \cap V) \rightarrow \mathbb{R}^m$. According to Definition \ref{def2}, it's a smooth mapping from the subcartesian space $\phi(U \cap V)$ to $\mathbb{R}^m$, where the differential structure of $\phi(U \cap V)$ is generated by restrictions of smooth functions on $\mathbb{R}^k$. It follows that for each $\phi(x) \in \phi(U\cap V)$, there exist a smooth mapping $F: \mathbb{R}^k \rightarrow \mathbb{R}^m$ and a neighborhood $\tilde W$ of $\phi(x)$ in $\phi(U\cap V)$, such that $F|\tilde W=\psi \circ \phi^{-1}|\tilde W$. Let $W=\phi^{-1}(\tilde W)$. Then the result follows immediately.
\end{proof}

\begin{lemma}\label{lem5}
Let $S$ be a subcartesian space with structural dimension $n$. Let $\Phi:S\rightarrow \mathbb{R}^m$ be a smooth mapping with $m\geq 2n$.
Then for any $\delta>0$, there exists a smooth mapping
$\Psi: S\rightarrow \mathbb{R}^m$ satisfying $\|\Phi(x)-\Psi(x)\|<\delta$ for every $x \in S$, where $\|.\|$ denotes the Euclidean norm. Besides, for every $x\in S$,  there exist some local chart $(U, \mathbb{R}^{k}, \phi)$ where $x \in U$, and a smooth extension $\Psi_x: \mathbb{R}^{k} \rightarrow \mathbb{R}^m$ of $\Phi$ at $x$, such  that  $\Psi_x$ is an immersion on a neighborhood $V$ of $\phi(x)$ in $\mathbb{R}^{k}$.
\end{lemma}
\begin{proof}
Let $(U_j)_{j\in \mathbb{Z}_{>0}}, (V_j)_{j \in \mathbb{Z}_{>0}}, (W_j)_{j \in \mathbb{Z}_{>0}}$ be the open covers
satisfying the conditions in Lemma \ref{lem2}. For each $j\in \mathbb{Z}_{>0}$, we claim that there exist smooth function $\rho_j: S\rightarrow [0,1]$ such that $\rho_j(x)=1$ for $x \in cl(U_j)$ and $\rho_j(x)=0$ for $x \in S/cl(V_j)$. Consider the open cover $\{V_j, S/cl(U_j)\}$.
Let $\{(\mathscr{V}_i, g_i)\}$ be a subordinate partition on unity and $f=\sum g_i$ where the sum is over those $i$ for which $\mathscr{V}_i\cap cl(U_j)\neq \emptyset$. Then $f$ is smooth, is one on $cl(U_j)$, and zero on $S/cl(V_j)$.

We inductively define a sequence of maps $\Phi_k: S \rightarrow \mathbb{R}^m$ such
that $\Phi_k-\Phi$ is smooth and has support contained in $K_k=\cup_{j=1}^kcl(W_j)$ for each $k\in \mathbb{Z}_{>0}$. For
$k=1$, let $A_1\in \mathbb{R}^{m\times n}$ and define
\[\Phi_1(x)=\Phi(x)+\rho_1(x)A_1\phi_1(x).\]
Since $\rho_1(x)=1$ for $x \in cl(U_1)$.  We have $\Phi_1(x)=\Phi(x)+A_1\phi_1(x)$ for $x\in U_1$. Let $f_1$ be the smooth extension
of $\Phi$ on $W_1$ according to Lemma \ref{lem2}. Then it follows from  Lemma \ref{lem4} that  we can choose  $A_1$ such that (1) the smooth extension $F(y)=f_1(y)+A_1y$ of $\Phi_1$ on $U_1$ is an immersion on $\mathbb{R}^{n_1}$. Hence for each $x \in cl(U_1) \subseteq V_1$, there exists a smooth extension of $\Phi_1$ at $x$ which is also an immersion; (2) $\|\Phi_1(x)-\Phi(x)\|<\frac{\delta}{2}$ for $x \in S$.

Now suppose that we have defined $\Phi_1, \cdots, \Phi_k$ such that, for
each $j\in \{1, \cdots, k\}, \Phi_j- \Phi$ is smooth, has support in $K_j$, and satisfies $\|\Phi_j(x)-\Phi_{j-1}(x)\|<\frac{\delta}{2^j}$
for $x\in S$. Besides, for each $x \in \cup_{i=1}^{k}cl(U_i)$, suppose that there exists a smooth extension of $\Phi_k$ at $x$ which is also an immersion.

For $A_{k+1}\in \mathbb{R}^{m\times n}$ denote
\[\Phi_{k+1}(x)=\Phi_k(x)+\rho_{k+1}(x)A_{k+1}\phi_{k+1}(x).\]
We claim that we can choose $A_{k+1}$ such that (1) there exists a smooth extension of $\Phi_{k+1}$ at $x$ which is also an immersion, for each $x\in \cup_{i=1}^{k+1}cl(U_{i})$;  (2) $\|\Phi_{k+1}(x)-\Phi_k(x)\|<\frac{\delta}{2^{k+1}}$
for $x \in S$.

To prove the above claim. Let $x \in (\cup_{i=1}^{k}cl(U_{i})) \cap cl(W_{k+1})$. Suppose that $x \in cl(U_i)$ for some $i$ satisfying $1\leq i\leq k$. Let $(W_i, \mathbb{R}^{n_i}, \phi_i)$ and $(W_{k+1}, \mathbb{R}^{n_{k+1}}, \phi_{k+1})$ be two local charts of $S$ containing $x$.  It follows from Lemma \ref{lem11} that there exist a neighborhood $W_x$ of $x$ and a smooth mapping $F_{i{k+1}}: \mathbb{R}^{n_i}\rightarrow \mathbb{R}^{n_{k+1}}$ such that $F_{i{k+1}}\circ \phi_i|W_x=\phi_{k+1}|W_x$. Hence we have
\[\Phi_{k+1}(x)=\Phi_k(x)+\rho_{k+1}(x)A_{k+1}(F_{i{k+1}}\circ \phi_i)(x),\]
for $x \in W_x$. Then by shrinking $W_x$ if necessary, it follows from Lemma \ref{lem3} that we can choose $A_{k+1}$ sufficiently small such that there exists a smooth extension of $\Phi_{k+1}$ on $W_x$ which is an immersion on a local neighborhood of $\phi_i(x)$ in $\mathbb{R}^{n_i}$ containing $\phi(W_x)$. Since $(\cup_{i=1}^{k}cl(U_{i})) \cap cl(W_{k+1})\subseteq cl(W_{k+1})$ and $cl(W_{k+1})$ is compact, it follows that $(\cup_{i=1}^{k}cl(U_{i})) \cap cl(W_{k+1})$ is compact. Hence it can be covered by finitely many $W_x$, for $x \in (\cup_{i=1}^{k}cl(U_{i})) \cap cl(W_{k+1})$. This means that we can choose $A_{k+1}$ such that for each $x \in (\cup_{i=1}^{k}cl(U_{i})) \cap cl(W_{k+1})$, there exist a local chart $(W, \mathbb{R}^k, \phi)$ and smooth extension of $\Phi_{k+1}$ at $x$, which is an immersion in a neighborhood of $\phi(x)$ in $\mathbb{R}^k$.
For $x\in S/cl(W_{k+1})$, we have $\Phi_{k+1}(x)=\Phi_k(x)$. It follows that for each $x \in (\cup_{i=1}^{k}cl(U_{i}))/cl(W_{k+1})$, there exists a smooth extension of $\Phi_{k+1}$ at $x$, which is an immersion. For $x \in cl(U_{k+1})$, the existence of smooth extension of $\Phi_{k+1}$ at $x$ that is an immersion is ensured by Lemma \ref{lem4}. Finally, we can choose $A_{k+1}$  sufficiently small such that $\|\Phi_{k+1}(x)-\Phi_k(x)\|<\frac{\delta}{2^{k+1}}$ for $x \in S$. This completes the proof for the above claim.

Let $\Psi(x)=\lim_{k \rightarrow \infty}\Phi_k(x)$, this limit existing and	$\Psi$ being smooth since, for each $x \in S$, according to Lemma \ref{lem2}, there exists $N \in \mathbb{Z}_{>0}$ such that $y \notin W_{N+1}, W_{N+2}, \cdots$, for each $y$ in some neighbourhood of $x$.
 Hence $\Phi_N(y)=\Psi(y)$, for $y$ in some neighbourhood of $x$. Besides, since $U_i \subseteq W_i, i=N+1, N+2, \cdots$, it follows that $x \in S/\cup_{i=N+1}^{\infty}U_{i}$, which yields that $x \in \cup_{i=1}^{N}U_{i}$.
It follows from the above construction of $\Phi_N$ that there exists a smooth extension of $\Psi$ at $x$ which is an immersion. And $\|\Phi-\Psi\|<\delta$ for every $x \in S$. This completes the proof of the theorem.
\end{proof}

\begin{lemma}\label{lem6}
Let $S$ be a subcartesian space with structural dimension $n$. Let $\Phi:S\rightarrow \mathbb{R}^m$ be a smooth mapping with $m\geq 2n+1$. Then for any $\delta>0$,  there exists a smooth injective mapping $\Psi: S\rightarrow \mathbb{R}^m$ satisfying $\|\Phi(x)-\Psi(x)\|<\delta$ for every $s \in S$. Besides, for every $x\in S$,  there exist some local chart $(U, \mathbb{R}^{l}, \phi)$ where $x \in U$, and a smooth extension $\Psi_x: \mathbb{R}^{l} \rightarrow \mathbb{R}^m$ of $\Phi$ at $x$, such  that  $\Psi_x$ is an immersion on a neighborhood  of $\phi(x)$ in $\mathbb{R}^{l}$.
\end{lemma}
\begin{proof}
According to Lemma \ref{lem5}, let $\Upsilon: S \rightarrow \mathbb{R}^{m}$ be a smooth mapping such that $\|\Upsilon(x)-\Phi(x)\|<\delta$ for every $x \in S$. Besides, for every $x\in S$,  there exist some local chart $(U, \mathbb{R}^{l}, \phi)$ where $x \in U$, and a smooth extension $\Upsilon_x: \mathbb{R}^{l} \rightarrow \mathbb{R}^m$ of $\Phi$ at $x$, such  that  $\Upsilon_x$ is an immersion on a neighborhood of $\phi(x)$ in $\mathbb{R}^{l}$. It follows that for each $x \in S$ there exists a  neighbourhood $U_x$ of $x$ such that $\Upsilon|U_x$ is injective. By using similar constructions as  given in Lemma \ref{lem2}, we can get a locally finite open cover $(U_j)_{j\in \mathbb{Z}_{>0}}$ of $S$ by relatively compact open sets such that
$\Upsilon|U_j$ is injective.

For each $j\in \mathbb{Z}_{>0}$, let $\rho_j$ be a smooth
function on $S$ taking values in $[0, 1]$ and satisfying $\rho_j(x) = 0$ for $x\in S/U_j$, and suppose
that $\sum_{j=1}^{\infty}\rho_j(x)=1$ for every $x \in S$, i.e., $(\rho_j)_{j\in \mathbb{Z}_{>0}
}$ is a partition of unity. For any sequence
$(y_j)_{j\in \mathbb{Z}_{>0}}$ in $\mathbb{R}^{m}$ satisfying $\|y_j\|<\frac{\delta}{2^{j+1}}$ for $x\in U_j$. for $k\in \mathbb{Z}_{>0}$ define
\[\Psi_k(x)=\Upsilon(x)+\sum_{j=1}^{k}\rho_j(x)y_j.\]
By Lemma \ref{lem3} above the sequence $(y_j)_{j\in \mathbb{Z}_{>0}}$ can be chosen such that for every $x\in S$,  there exist some  local chart $(U, \mathbb{R}^l, \phi)$ and  a smooth extension $\Psi_{k_x}:\mathbb{R}^l \rightarrow \mathbb{R}^m$ of $\Psi_k$ at $x$,  satisfying that  $\Psi_{k_x}$ is an immersion on a neighborhood of $\phi(x)$ in $\mathbb{R}^l$,  for each $k\in  \mathbb{Z}_{>0}$.

For each $k\in \mathbb{Z}_{>0}$, let
\[D_k=\{(p,q) \in S \times S|\rho_k(p)\neq \rho_k(q)\}\]
and define the mapping $\theta_k: D_k \rightarrow \mathbb{R}^m$ by
\[\theta_k(x,y)=\frac{\Psi_{k-1}(x)-\Psi_{k-1}(y)}{\rho_k(x)-\rho_k(y)}\]
Since ${\theta_k(D_k)}$ has measure zero ($m \geq 2n+1$),  the sequence $(y_j)_{j\in \mathbb{Z}_{>0}}$ can be chosen by additionally asking that
$y_k \notin \theta_k(D_k)$. We then define $\Psi(x)=\lim_{k \rightarrow \infty}\Psi_k(x)$. It follows that $\Psi$ is a smooth mapping satisfying that $\|\Psi(x)-\Phi(x)\|<\delta(x)$. Besides, for every $x\in S$,  there exists a smooth extension $\Psi_x: \mathbb{R}^l \rightarrow \mathbb{R}^m$ of $\Psi$ at $x$, with local charts $(U, \mathbb{R}^l, \phi)$,  satisfying that  $\Psi_x$ is an immersion on a neighborhood of $\phi(x)$ in $\mathbb{R}^l$.

It remains to show that $\Psi$ is injective. Suppose that $\Psi(x)=\Psi(y)$. Let $N \in \mathbb{Z}_{>0}$ be
sufficiently large that $\rho_{j+1}(x)=\rho_{j+1}(y)$ and $\Psi_j(x)=\Psi_j(y)=\Psi(x)=\Psi(y)$ for every $j\geq N$.
Since $y_{N-1} \notin \theta_{N-1}(D_{N-1})$ we have $\rho_{N-1}(x)=\rho_{N-1}(y)$. Otherwise, since $\Psi_{N-1}(x)+\rho_{N}(x)y_{N}=\Psi_{N}(x)=\Psi_{N}(y)=\Psi_{N-1}(y)+\rho_{N}(y)y_{N}$, we have
$y_{N}=\frac{\Psi_{N-1}(x)-\Psi_{N-1}(y)}{\rho_{N}(x)-\rho_{N}(y)}\in \theta_{N}(D_{N})$. This makes contradiction.

Consequently, we have $\Psi_{N-1}(x)=\Psi_{N-1}(y)$. We can proceed inductively backwards to
conclude that $\rho_j(x)=\rho_j(y)$ and	$\Psi_j(x)=\Psi_j(y)$ for every $j\in \mathbb{Z}_{>0}$. Consequently, $\Upsilon(x)=\Upsilon(y)$.
Since $\rho_j(x) \neq 0$ for at least one $j \in  \mathbb{Z}_{>0}$, it follows that $x,y \in U_j$ for some $j \in \mathbb{Z}_{>0}$. However,
this is in contradiction with the fact that $\Upsilon|U_j$ is injective. This completes the proof of the above theorem.
\end{proof}

\begin{definition}
Let $S$ be  a topological space. A function $u: S\rightarrow [-\infty,\infty)$
is an exhaustion function for $S$ if the sublevel set $u^{-1}([-\infty, a))$ is a relatively
compact subset of $S$ for every $a \in \mathbb{R}$.
\end{definition}
\begin{lemma}\label{lem7}
A continuous map $u: S\rightarrow [-\infty, \infty)$ is an exhaustion function for a topological space $S$ if and only if, for any
$a\in \mathbb{R}$, there exists a compact set $K \subseteq S$ such that $u(x)>a$ for every $x \in S/K$.
\end{lemma}
\begin{proof}
First suppose that $u$ is an exhaustion function and let $a \in \mathbb{R}_{>0}$. Since $u$ is an
exhaustion function, $u^{-1}([-\infty, a])$ is compact. Moreover, $u(x)>a$ for every $x\in S/u^{-1}([-\infty, a])$.
For the converse, suppose that, for any $a \in \mathbb{R}_{>0}$, there exists a compact set $K \subseteq S$
such that $u(x)>a$ for every $x \in S/K$. Let $a \in \mathbb{R}$ and note that $u^{-1}([-\infty, a])$ is compact,
being a closed subset of a compact set, showing that $u$ is an exhaustion function.
\end{proof}

\begin{lemma}\label{lem8}
Let $S$ be a subcartesian space. Then there exists a smooth exhaustion function on $S$.
\end{lemma}
\begin{proof}
Since $S$ is locally compact, Hausdorff and second countable, it possesses a countable open cover by relatively compact open sets.
According to Theorem \ref{thm2}, there exists a partition of unity $\{\rho_i\} \in C^{\infty}(S)$ subordinate to such an open cover.
Consider the function $f=\sum_{j=1}^{\infty}j\rho_j$. It's easy to see that $f \in C^{\infty}(S)$. We claim that $f$ is an exhaustion function.
Indeed, let $a \in \mathbb{R}$ and let $N \in \mathbb{Z}_{>0}$ be such that $N>a$. Let $K$ be a
compact set containing the supports of $\rho_1, \cdots, \rho_N$. We claim that $f^{-1}((-\infty, a])\subseteq K$. Indeed,
suppose that $x \notin K$, then we have $x \notin \cup_{j=1}^N \text{supp}(\rho_j)$. It follows that $\sum_{j=N+1}^{\infty}\rho_j(x)=\sum_{j=1}^{\infty}\rho_j(x)=1$. Hence $f(x)=\sum_{j=1}^{\infty}j\rho_j(x)=\sum_{j=N+1}^{\infty}j\rho_j(x)
> N\sum_{j=N+1}^{\infty}\rho_j(x)=N>a$. By Lemma \ref{lem7} we know that $f$ is an exhaustion function on $S$. Hence  the result follows immediately.
\end{proof}

\begin{definition}
For  topological space $S$ and $T$, a continuous mapping $f: S \rightarrow T$ is proper if $f^{-1}(K)$ is compact for every compact subset $K \subseteq T$.
\end{definition}
\begin{lemma}\label{lem12}\cite{8}
Let $S$ be a Hausdorff topological space and $T$ be a locally compact Hausdorff topological space. Let $f: S \rightarrow T$ be a proper injective mapping. Then $f: S \rightarrow f(S)$ is a homeomorphism provided that $f(M)$ is equipped with the subspace topology inherited from $T$.
\end{lemma}
The following result indicates embedding theorem for subcartesian space.
\begin{theorem}\label{thm1}
Let $S$ be a subcartesian space with structural dimension $n$. Then $S$ is diffeomorphic to a subset of $\mathbb{R}^m$, where $m\geq 2n+1$.
\end{theorem}
\begin{proof}
According to Lemma \ref{lem8}, let $u$ be a smooth exhaustion function on $S$.
Now define $\Phi: S\rightarrow \mathbb{R}^m$ by
\[\Phi(x)=(u(x), 0, \cdots, 0).\]
By Lemma  \ref{lem6}, let $\Psi: S \rightarrow \mathbb{R}^m$ be a smooth injective mapping satisfying that $\|\Psi(x)-\Phi(x)\| < 1$ for
every $x\in S$. Besides, for every $x\in S$,  there exists a smooth extension $\Psi_x: \mathbb{R}^l \rightarrow \mathbb{R}^m$ of $\Psi$ at $x$, with local charts $(U, \mathbb{R}^l, \phi)$,  satisfying that  $\Psi_x$ is an immersion on a neighborhood of $\phi(x)$ in $\mathbb{R}^l$.

We claim that $\Psi$ is proper. Let $K\subseteq \mathbb{R}^m$ be compact. If $x \in \Psi^{-1}(K)$ then
\[\|\Psi(x)\|\leq \text{sup}\{\|x\||x\in K\}=C_k,\]
and so $\|\Phi(x)\| \leq C_K+1$. Since $\|\Phi(x)\|=u(x)$, we have \[\Psi^{-1}(K)\subseteq u^{-1}([-\infty, C_K+1]).\] Since
$u$ is an exhaustion function we conclude that 	$\Psi^{-1}(K)$ is compact, being a closed subset of
the compact set $u^{-1}([-\infty, C_K+1])$. It follows from Lemma \ref{lem12} that $\Psi: S \rightarrow \Psi(S)$ is a homeomorphism.

It remains to show that: (1) for any  $f \in C^{\infty}(\mathbb{R}^m)$, $\Psi^*f \in C^{\infty}(S)$; (2) for any $g \in C^{\infty}(S)$, $(\Psi^{-1})^*g \in C^{\infty}(\Psi(S))$, with the differential structure $C^{\infty}(\Psi(S))$ generated by restrictions of smooth differential structure on $\mathbb{R}^m$.

To prove (1), let $\Psi=(\Psi_1, \cdots, \Psi_m)$. Since $\Psi$ is smooth, it follows that $\Psi_i \in C^{\infty}(S)$, for $i=1, \cdots, m$. Since $f \in C^{\infty}(\mathbb{R}^m)$, it follows from condition 2 in Definition 1 that $\Psi^*f=f(\Psi_1, \cdots, \Psi_m) \in C^{\infty}(S)$.
To prove (2), it follows from Lemma \ref{lem6} that for each $x \in S$, there exists a smooth extension $\Psi_x: \mathbb{R}^l\rightarrow \mathbb{R}^m$ of $\Psi$ at $x$, with local chart $(U, \mathbb{R}^l, \phi)$, such that $\Psi_x$ is an immersion on a neighborhood $V$ of $\phi(x)$ in $\mathbb{R}^l$. By shrinking $V$ if necessary, there exists a smooth function $g_x$ on $\mathbb{R}^l$, such that $g|V\cap \phi(U)=g_x\circ \phi|\phi^{-1}(V)\cap U$. Due to the local representatives for the immersion $\Psi_x$,  there exist a neighborhood $W$ of $\Psi(x)$ in $\mathbb{R}^m$,  and $h\in C^{\infty}(W)$, such that
$h|W\cap \Psi(S)=(\Psi^{-1})^*g|W\cap \Psi(S)$. By using bump function, we know that for each $y \in \Psi(S)$, there exists a function $h_y \in C^{\infty}(\mathbb{R}^m)$, such that $(\Psi^{-1})^*g|W_y=h_y|W_y$, where $W_y$ is a neighborhood of $y$ in $\Psi(S)$. It follows that
$(\Psi^{-1})^*g \in C^{\infty}(\Psi(S))$. This completes the proof of (2). Hence the result of the theorem follows immediately.
\end{proof}

\begin{corollary}(The Whitney Embedding Theorem)
If $M$ is a smooth, paracompact, Hausdorff and connected manifold of dimension $n$, then there exists a proper smooth embedding of $M$ in
$\mathbb{R}^{2n+1}$.
\end{corollary}
\begin{proof}
Since a smooth manifold is a subcartesian space with differential structure given by all the smooth functions on $M$, it follows from Theorem \ref{thm1} that there exists a proper injective mapping $\Psi=(\Psi_1, \cdots, \Psi_{2n+1}): M \rightarrow \mathbb{R}^{2n+1}$ that is smooth in the sense of differential spaces. Let $(x_1, \cdots, x_{2n+1})$ be coordinates for $\mathbb{R}^{2n+1}$, it follows that $\Psi_i=\Psi^*x_i$  is a smooth function on $M$. That is, $\Psi$ is smooth in the sense of manifolds. From the proof of Lemma \ref{lem5} and Lemma \ref{lem6} it can also be ensured that $\Psi$ is an immersion. Lemma \ref{lem12} ensures that $\Psi: M \rightarrow \Psi(M)$ is a homeomorphism. It follows that $\Psi$ is an embedding. Hence the result follows immediately.
\end{proof}
\section{Smooth generalized distributions on subcartesian spaces}
\label{sec:3}
\begin{theorem}\cite{5}\label{thm4}
Let $S$ be a differential subspace of $\mathbb{R}^n$, and let $X$ be a derivation of $C^{\infty}(S)$. For each $x \in S \subseteq \mathbb{R}^n$,
there exist a neighborhood $U$ of $x$ in $\mathbb{R}^n$ and a vector field $Y$ on $\mathbb{R}^n$ such that
\begin{equation}\label{eq:1}
X(F|S)|U\cap S=(Y(F))|U\cap S
\end{equation}
for every $F \in C^{\infty}(\mathbb{R}^n)$.
\end{theorem}
\begin{theorem}\cite{3}\label{thm5}
Let $M$ be a smooth connected manifold and let $E$ be a smooth vector bundle over $M$. Let $F$ be a smooth generalized subbundle of $E$.
Then $F$ is globally finitely generated.
\end{theorem}
\begin{theorem}\label{thm6}
Let $S$ be a subcartesian space. Let $D$ be a smooth generalized distribution on $S$. Then $D$ is globally finitely generated.
\end{theorem}
\begin{proof}
According to Theorem \ref{thm1}, $S$ is diffeomorphic to a subset $\bar S$ of $\mathbb{R}^m$. It remains to show that for any smooth generalized distribution $D$ on a subset $\bar S$ of $\mathbb{R}^m$, $D$ is globally finitely generated. We first construct  a neighborhood of $\bar S$ in  $\mathbb{R}^m$ with a smooth generalized distribution on it as follows: according to Definition \ref{def1}, for each $x \in \bar S$ and $v_x\in D_x$, by using bump function, there exist a derivation $X_x$ of $C^{\infty}(\bar S)$, such that $X_x$ is a smooth section of $D$ and $X_x(x)=v_x$. It follows from Theorem \ref{thm4} that there exist a neighborhood $U_{v_x}$ of $x$ and a vector field $Y_{v_x}$ on $\mathbb{R}^n$ such that
\begin{equation}\label{eq:3}
X_x(F|\bar S)|U_{v_x}\cap \bar S=(Y_{v_x}(F))|U_{v_x}\cap \bar S,
\end{equation}
for every $F \in C^{\infty}(\mathbb{R}^n)$. Now consider the open subset $\cup_{x\in \bar S}\cup_{v_x \in D_x}U_{v_x}$ of $\mathbb{R}^m$.  For each $y \in \cup_{x\in \bar S}\cup_{v_x \in D_x}U_{v_x}$, let $\Lambda$ be the set of all points in $T\bar S$ such that $U_v$ contains $y$ for every $v \in \Lambda$.
We define a generalized distribution $\hat D$ on $\cup_{x\in \bar S}\cup_{v_x \in D_x}U_{v_x}$ by
\begin{equation}\label{eq:4}
\hat D_y=\text{span}_{\mathbb{R}}\{Y_v(y)|v \in \Lambda\}.
\end{equation}
It's easy to see that $\hat D$ is smooth on the open submanifold $\cup_{x\in \bar S}\cup_{v_x \in D_x}U_{v_x}$ of $\mathbb{R}^m$. Then according to Theorem \ref{thm5}, there exist smooth sections $Y_1, \cdots, Y_l$  of $\hat D$ such that $\hat D_y$ is spanned by $Y_1(y), \cdots, Y_l(y)$ for every $y \in  \cup_{x\in \bar S}\cup_{v_x \in D_x}U_{v_x}$. We claim that $Y_1, \cdots, Y_l$ yield derivations $X_1, \cdots, X_l$ of $C^{\infty}(\bar S)$ that span $D_x$ for each $x\in \bar S$.

For each $f \in C^{\infty}(\bar S)$, for each $x \in \bar S$, there exist a neighborhood $U$ of $x$ in $\cup_{x\in \bar S}\cup_{v_x \in D_x}U_{v_x}$, and a function $F_x \in C^{\infty}(\cup_{x\in \bar S}\cup_{v_x \in D_x}U_{v_x})$, such that $f|U\cap \bar S=F_x|U\cap \bar S$. Set $X_i(f)|U\cap \bar S=Y_i(F_x)|U\cap \bar S$. We claim that $X_i(f)\in C^{\infty}(\bar S)$ is well-defined.
Let $V$ be another neighborhood of $x$ in $\cup_{x\in \bar S}\cup_{v_x \in D_x}U_{v_x}$, and let $H_x\in C^{\infty}(\cup_{x\in \bar S}\cup_{v_x \in D_x}U_{v_x})$ be a function such that $f|V\cap \bar S=H_x|V\cap \bar S$.
We have that $F_x|U\cap V\cap \bar S=H_x|U\cap V\cap \bar S$. Then it follows from Lemma \ref{lem15}, (\ref{eq:3}) and (\ref{eq:4}) that
$Y_i(F_x-H_x)|U\cap V\cap \bar S=0$. This yields that $X_i(f)$ is a well-defined smooth function on $\bar S$. Hence we get a derivation $X_i$ of $C^{\infty}(\bar S)$.
From (\ref{eq:3}) and (\ref{eq:4}) we can also see that $X_i(x) \in D_x$ for each $x \in \bar S$.

It remains to show that $X_1(x), \cdots, X_l(x)$ span $D_x$ at each $x \in \bar S$. Let $w_x\in D_x$. Then according to the above construction of
$\hat D$, we know that there exist a derivation $X$ of $C^{\infty}(\bar S)$ such that $X(x)=w_x$, a neighborhood $U_{w_x}$ of $x$ and a vector field $Y_{w_x}$ on $\mathbb{R}^m$ such that
$X(F|\bar S)|U_{w_x}\cap \bar S=(Y_{w_x}(F))|U_{w_x}\cap \bar S$, for every $F \in C^{\infty}(\mathbb{R}^m)$, where
$Y_{w_x}(x)\in \bar D_x$. So there exist constants $c_1, \cdots, c_l$ such that $Y_{w_x}(x)=\sum_{i=1}^lc_iY_i(x)$. Then it follows from the above definition of $X_i$ that $\sum_{i=1}^lc_iX_i(x)=w_x$. This completes the proof for the above claim. Hence the result of the theorem follows immediately.
\end{proof}
\section{Smooth generalized subbundles on subcartesian spaces}
\label{sec:4}
\begin{lemma}\label{lem14}
Let $S$ be a subcartesian space. Let $(E, \pi, S, \mathbb{R}^k)$ be a smooth vector bundle. Then $E$ is globally finitely generated.
\end{lemma}
\begin{proof}
It follows from Proposition \ref{prop2} that $(E, \pi, S, \mathbb{R}^k)$ admits a finite coordinate representation $\{(U_i, \psi_i)|i=1, \cdots, m\}$.
Then on each $U_i$, there exist local sections $\gamma_{i1}, \cdots, \gamma_{ik}$ of $(E, \pi, S, \mathbb{R}^k)$  such that $\pi^{-1}(x)=\text{span}_{\mathbb{R}}\{\gamma_{i1}(x), \cdots, \gamma_{ik}(x)\}$ for each $x \in U_i, i=1, \cdots, m$. Let $\{\rho_j|j=1, 2,\cdots\}$ be the partition of unity for $S$ subordinate to the covering $\{U_i|i=1, \cdots, m\}$. For $i\in \{1, 2, \cdots, m\}$, let \[I_i=\{l \in \{1,2, \cdots\}|\text{the support of}\, \rho_l \,\text{is contained in} \,U_i\}.\]
Define global sections $\sigma_{ij}$ by
\[\sigma_{ij}=\sum_{l\in I_i}\rho_{l}\gamma_{ij},\]
for $i=1, \cdots, m$ and  $j=1, \cdots, k$.  Then for each $e_x \in \pi^{-1}(x)$, there exists some $\rho_i$ such that $\rho_i(x)\neq 0$ and the support of $\rho_i$ is contained in some $U_j$. It follows that there exist constants $c_1, \cdots, c_k$ such that $e_x=\sum_{l=1}^kc_l\sigma_{jl}(x)$. Hence the result follows.
\end{proof}
\begin{proposition}\label{prop3}
For every smooth vector bundle $(E, \pi, S, \mathbb{R}^k)$ there exist a trivial bundle $S \times \mathbb{R}^l$, for some $l \in \mathbb{N}$ and a
smooth injective bundle map $\varphi: E \rightarrow S \times \mathbb{R}^l$.
\end{proposition}
\begin{proof}
According to  Lemma \ref{lem14}, we know that  there exist finite global generators $\sigma_1, \cdots, \sigma_{l}$ for $E$. Consider the trivial
bundle $S \times \mathbb{R}^{l}$ and the bundle map $\psi: S\times \mathbb{R}^l \rightarrow E$ given by
\[\psi(x, \sum_{i=1}^l\lambda_ie_i)= \sum_{i=1}^l\lambda_i\sigma_i(x), x \in S,\]
where $e_1, \cdots, e_l$ is a basis for $\mathbb{R}^l$. Then each linear map $\psi_x:\mathbb{R}^l\rightarrow \pi^{-1}(x)$ is surjective. From Proposition \ref{prop4}, there exist smooth Riemannian metrics on $(E, \pi, S, \mathbb{R}^k)$ and $(S\times \mathbb{R}^l, \pi, S, \mathbb{R}^l)$.
Define $\varphi: E\rightarrow S \times \mathbb{R}^l$ by $\varphi_x=\psi_x^*(\psi_x\circ\psi_x^*)^{-1}, x\in S$, where $\psi_x^*$ is the adjoint map of $\psi_x$ associated with the smooth Riemannian metrics on $(E, \pi, S, \mathbb{R}^k)$ and $(S\times \mathbb{R}^l, \pi, S, \mathbb{R}^l)$.
$\varphi$ is a smooth bundle map satisfying that $\psi\circ \varphi=\text{identity}$. Hence $\varphi$ is injective.
\end{proof}
\begin{theorem}
Let $S$ be a  subcartesian space. Let  $(E, \pi, S, \mathbb{R}^k)$  be a smooth vector bundle. Let $F$ be a smooth generalized subbundle of $E$. Then $F$ is globally finitely generated.
\end{theorem}
\begin{proof}
According to Proposition \ref{prop3}, it suffices to show that every smooth generalized subbundle of the trivial bundle $S\times \mathbb{R}^l$ is globally finitely generated. We identify each
subspace $F(\pi^{-1}(x))$ with a subspace of $\mathbb{R}^l$ and identify sections of $S\times \mathbb{R}^l$, and hence sections of
$F$, with $R^l$-valued functions. Further, it follows from Theorem \ref{thm1} that $S$ is diffeomorphic to a subset of $\mathbb{R}^m$.
Hence for each $e_x \in F(\pi^{-1}(x))$, since $F$ is smooth, there exist a neighborhood $U_{e_x}$ of $x$ in  $\mathbb{R}^m$ and a smooth section $\xi_{U_{e_x}}$ satisfying that $\xi_{U_{e_x}}(x)=e_x$ and $\xi_{U_{e_x}}|U_{e_x}\cap S$ is a local smooth section of $F$. Now consider the open subset
$U=\cup_{x\in S}\cup_{e_x\in F(\pi^{-1}(x))}U_{e_x}$ and the  generalized subbundle $\hat F$ of $U \times \mathbb{R}^l$ defined by
\[\hat F(y)=\text{span}_{\mathbb{R}}\{\xi_{U_{e_x}}(y)|y \in U_{e_x}\}.\]
It's obvious that $\hat F$ is smooth. It follows from Theorem \ref{thm5} that there exist smooth global generators $\xi_1, \cdots, \xi_s$ of $\hat F$. We claim that $\xi_1|S, \cdots, \xi_s|S$ are smooth global generators of $F$. It follows from the definition of $\xi_{U_{e_x}}$ and $\hat F$ that
$\xi_1|S, \cdots, \xi_s|S$ are smooth sections of $F$. Besides,  since $e_x \in \hat F(x)$ for each $x \in S$ and $e_x \in F(x)$, it follows immediately that $\xi_1|S, \cdots, \xi_s|S$ are global generators of  $F$. This completes the proof of the theorem.
\end{proof}



\end{document}